\documentclass[a4paper,reqno,12pt,oneside]{amsart}
\usepackage{a4wide}
\usepackage[english]{babel}    
\usepackage[bibliography=common]{apxproof}
\usepackage[ruled, linesnumbered]{algorithm2e}
\usepackage[latin1]{inputenc} 
\usepackage[sorting=none, style=numeric]{biblatex}
\usepackage[page]{appendix}
\usepackage{graphicx}
\usepackage{tikz}
\usepackage{subfig}
\usetikzlibrary{positioning}
\newtheorem{definition}{Definition}[section]
\newtheorem{proposition}{Proposition}[section]
\newtheorem{theorem}{Theorem}[section]
\newtheorem{corollary}{Corollary}[theorem]
\theoremstyle{remark}
\newtheorem*{remark}{Remark}
\newtheorem{lemma}{Lemma}[section]
\usepackage[centertags]{amsmath}
\usepackage{indentfirst,amsfonts,amssymb,amsthm,newlfont}
\makeatletter

\title[The Bell-Touchard counting process]{The Bell-Touchard counting process}

\author[Thomas Freud]{Thomas Freud}
\email{thfreud@gmail.com}
\author[Pablo M. Rodriguez ]{Pablo M. Rodriguez}

\address{Centro de Ci\^encias Exatas e da Natureza, Universidade Federal de Pernambuco, Av. Prof. Moraes Rego, 1235 - Cidade Universit\'aria - Recife - PE, Brazil.}
\email{pablo@de.ufpe.br}

\begin{document}


\maketitle

\begin{abstract}The Poisson process is one of the simplest stochastic processes defined in continuous time, having interesting mathematical properties, leading, in many situations, to applications mathematically treatable. One of the limitations of the Poisson process is the rare events hypothesis; which is the hypothesis of unitary jumps within an infinitesimal window of time. Although that restriction may be avoided by the compound Poisson process, in most situations, we don't have a closed expression for the probability distribution of the increments of such processes, leaving us options such as working with probability generating functions, numerical analysis and simulations. It is with this motivation in mind, inspired by the recent developments of discrete distributions, that we propose a new counting process based on the Bell-Touchard probability distribution, naming it the Bell-Touchard process. We verify that the process is a compound Poisson process, a multiple Poisson process and that it is closed for convolution plus decomposition operations. Besides, we show that the Bell-Touchard process arises naturally from the composition of two Poisson processes. Moreover, we propose two generalizations; namely, the compound Bell-Touchard process and the non-homogeneous Bell-Touchard process, showing that the last one arises from the composition of a non-homogeneous Poisson process along with a homogeneous Poisson process. We emphasize that since previous works have been shown that the Bell-Touchard probability distribution can be used quite effectively for modelling count data, the Bell-Touchard process and its generalizations may contribute to the formulation of mathematical treatable models where the rare events hypothesis is not suitable.
\end{abstract}

\section{Introduction}
The Poisson process is one of the simplest stochastic processes defined in continuous time, having interesting mathematical properties, leading, in many situations, to applications mathematically treatable. It is widely explored in the literature, having many books approaching it, such as \cite{kingman1992poisson,last2017lectures}. One can check the properties of this process and some applications in queuing theory and biology in \cite{ross2010introduction,schinazi2014classical}. In actuarial science, one applies the Poisson process to the classical risk process\cite{asmussen2010ruin,buhlmann2007mathematical}. There are some generalizations of the Poisson process one gets relaxing its original hypotheses. If one drops the stationary and independent increments hypothesis, one obtains the non-homogeneous Poisson process. On the other hand, releasing the assumption of unitary jumps within a small interval leads to, for example, a compound Poisson process. We also have the spatial Poisson process, defined in more than one dimension\cite{ferrari1997acoplamento}\cite[Chapter~11]{privault2013understanding}, which agrees with the original Poisson process in the real line.

The compound Poisson process is one of the well-known generalizations of the Poisson process, having lots of applications. In queuing theory, it appears naturally in batch arriving queue models\cite[Chapter~6]{balakrishnan2007statistics}. In actuarial finance, it comes forth both in ruin theory as a part of the Crame\'r-Lundberg process and its generalizations, see \cite[Chapter~4]{asmussen2010ruin}, as well as in stochastic interest modelling\cite{li2017stochastic}.

The Poisson distribution and the compound Poisson process are related to the Bell polynomials, which, for any $n \in \mathbb{N}_{0}:=\mathbb{N}\cup \{0\}$ and $x \in \mathbb{R}$, are defined as:
    \begin{equation}
        B_{n}(x) = e^{-x}\sum_{k=0}^{\infty}\frac{k^n}{k!}x^k.
        \label{eq:dobinski_formula}
    \end{equation}

The Bell polynomials occur in combinatorics when one deals with the study of set partitions. They also come about in the mathematical analysis, in the successive differentiation of composite functions\cite{boyadzhiev2009exponential,noschese2003differentiation}. Concerning the Poisson distribution, one has that its moment generating function is the same as the generating function of the complete Bell polynomials. That indicates equivalence between those polynomials and the moments of the Poisson distribution. This generating function is given by:
    \begin{equation}
            \varphi_{x}(\theta) := e^{x(e^{\theta}-1)} = \sum_{n=0}^{\infty}B_{n}(x)\frac{\theta^{n}}{n!}, \quad\text{ for } \theta>0,
            \label{eq:bell_generating_function}
    \end{equation}
where, differentiating \eqref{eq:bell_generating_function} successively and evaluating each derivative at $\theta=0$, one recovers the Bell polynomials. Moreover, expanding the left-hand side of \eqref{eq:bell_generating_function} and comparing each term with the right-hand side, one gets \eqref{eq:dobinski_formula}. Furthermore, the Bell polynomials have many applications in Probability Theory. Recently, \cite{kataria2021probabilistic} proposed an association between these polynomials and the weighted sum of independent Poisson random variables. Besides, the iterated Poisson process obtained from the composition of Poisson processes has the distribution of increments provided as a function of the Bell polynomials\cite{orsingher2012compositions}. The asymptotic properties of the Bell polynomials involving random graphs are explored in \cite{khorunzhiy2020asymptotic}. Probabilistic characteristics of these polynomials are employed by \cite{eger2016identities} in order to demonstrate some properties of the partial Bell polynomials.


One of the Poisson process assumptions specifies that the size of the jumps is unitary. However, \cite{janossy1950composed} relaxes this hypothesis, bringing about the multiple Poisson process. Additionally, this process has a distribution of increments given by the so called composed Poisson distribution, which is parametrized by a non-negative convergent sequence $\{c_{n}\}_{n\geq1}$. Considering a multiple Poisson process $\{\xi(t), \,t\geq 0\}$ depending on the non-negative sequence ${\{c_{n}\}_{n\geq1}}$ where $\sum c_n<\infty$ one can write the following:
	\begin{equation}
	   \xi(t) = \sum_{n=1}^{\infty} n\xi_n(t),
	  \label{eq:MultiplePoisson}
	\end{equation}
where $\{ \xi_{n}\}_{n \geq 1}$ is a sequence of independent Poisson processes with rate $c_{n}$ for all $n\in \mathbb{N}$. In other words, the multiple Poisson processes correspond to the weighted sum of an infinity number of Poisson processes, where the random variable $\xi(t)$ counts the number of events within the interval $[0,t)$. Moreover, it has probability distribution given by (see \cite{janossy1950composed}):
    \begin{equation*}
	        \Pr[\xi(t)=k] =\exp{\left \{-\sum_{n=1}^{\infty}tc_n \right \}} \sum_{r_1+2r_2+\dots+kr_k=k}\frac{(tc_1)^{r_1}(tc_2)^{r_2} \dots (tc_k)^{r_k}}{r_1!r_2!\dots r_k!}.
	        \label{eq:poisson_composed_family}
	\end{equation*}

The Bell distribution, proposed by \cite{castellares2018bell}, is a member of the composed Poisson family. It has probability distribution given by:
     \begin{equation*}
        \Pr[X=x] = e^{-(e^{\theta}-1)}\frac{\theta^{x}}{x!}B_{x}, \quad x \in \mathbb{N}_{0}\text{ and } \theta>0,
    \end{equation*}
where $B_{n}:=B_{n}(1)$ is the $\text{n}^\text{th}$ Bell number. The Bell distribution has interesting properties, such as being uniparametric, infinitely divisible, and a member of the exponential family of distributions \cite{castellares2018bell}. The convolution of independent Bell random variables follows the Bell-Touchard distribution, which was proposed by \cite{castellares2020two}, having the following probability distribution:
  \begin{equation*}
     \Pr[X=x] = e^{-\alpha(e^{\theta}-1)}\frac{\theta^{x}}{x!}B_{x}(\alpha), \quad x \in \mathbb{N}_{0}\text{ and } \theta,\alpha>0,
     \label{eq:bell_touchard_distribution}
 \end{equation*}
where $B_{n}(x)$ is the $\text{n}^\text{th}$ single variable Bell polynomial. This is equivalent to the Touchard polynomial, hence the proposed name for the distribution. Among its properties, this distribution is a member of the composed Poisson family, closed for convolutions and is infinitely divisible. Subsequently, \cite{bhatifamily} has proposed a distribution that generalizes the Bell-Touchard distribution taking the Taylor series expansion of the $r$-Bell polynomials generating function. It is the three parameters $r$-Bell distribution. Its probability mass function is given by:
    \begin{equation*}
         \Pr[X=x] = e^{-\alpha(e^{\theta}-1)-r\theta}\frac{\theta^{x}}{x!}B_{x,r}(\alpha), \quad x\in \mathbb{N}_{0}\text{ and } \theta,\alpha, r>0,
     \end{equation*}
     where
     \begin{equation*}
         B_{x,r}(\alpha) =\frac{1}{e^{\alpha}}\sum_{k=0}^{\infty}\frac{(k+r)^{x}}{k!}\alpha^{k}.
     \end{equation*}
The generating function of the $r$-Bell polynomials is defined by
     \begin{equation*}
         \varrho_{x}(\theta):=e^{x(e^{\theta}-1)+r\theta} = \sum_{n=0}^{\infty}B_{n,r}(x)\frac{\theta^{n}}{n!}.
     \end{equation*}
Considering two independent random variables, the first one following a Bell-Touchard distribution and the second one following a Poisson distribution, its sum follows the $r$-Bell distribution. The $r$-Bell is the same as the Short distribution\cite{irwin1964personal,kemp1967contagious}. In its turn, the Bell-Touchard distribution is equivalent to the Neyman Type A distribution, with $\alpha:=\lambda e^{-\theta}$ and $\theta$. That can be observed from the probability mass function of a Neyman Type A distribution, which is given by:
    \begin{equation*}
         \Pr[X=x] = \theta^{x}\frac{e^{-\lambda}}{x!}\sum_{k=0}^{\infty}\frac{(\lambda e^{-\theta})^{k}}{k!}k^{x}, \text{ } \theta>0, x\in \mathbb{N}_0.
     \end{equation*}
By using \eqref{eq:dobinski_formula} one gets
        \begin{equation*}
            \Pr[X=x] = e^{-\lambda e^{-\theta}(e^{\lambda}-1)}\frac{\theta^{x}}{x!}B_{x}(\lambda e^{-\theta}), \text{ } \theta>0, x\in \mathbb{N}_0,
        \end{equation*}
which is the Bell-Touchard probability mass function.

Since the distributions we have covered here are members of the composed Poisson family, one can write these random variables as \eqref{eq:MultiplePoisson}, making a suitable choice of the terms in the sequence ${\{c_{n}\}_{n\geq1}}$. Besides, these distributions are particular cases of the compound Poisson distribution. This fact is a consequence of the connection between the Bell polynomials and the compound Poisson family\cite{rota1973foundations}.

It has been shown that these new probability distributions can be used quite effectively for modeling count data, see for example \cite{castellares2020two,bhatifamily}. With this motivation in mind, and inspired in the mentioned works, we propose and investigate a counting process whose underlying distribution is Bell-Touchard with parameters $\alpha, \theta$, exploring some of its major properties. We point out that the new process, that we call the Bell-Touchard process, and its generalizations, may contribute with the formulation of mathematical treatable models where the rare events hypothesis is not suitable.

The paper is organized as follows. Section 2 is devoted to preliminary definitions and results related to the Bell polynomials and the Bell-Touchard distribution. The formulation of the Bell-Touchard process as well as its main properties and generalizations are presented in Section 3. In Section 4 we add some numerical results to illustrate the applicability of the process, and we finish the paper with a Concluding Remarks Section.


\section{Preliminaries}

\subsection{Bell polynomials}
In this section, we recap some properties of the Bell polynomials. Other properties can be checked in \cite{bell1934exponential, bell1934polynomials,mihoubi2008bell, san2015some}. For any $k \in \mathbb{N}_{0}$ and $x \in \mathbb{R}$, the single variable Bell polynomials can be expressed by the Fa\`a di Bruno's formula, besides \eqref{eq:dobinski_formula}:
\begin{equation}
        B_0(x):=1, \quad B_k(x) := \sum_{r_1+2r_2+\dots+kr_k=k}\frac{k!}{r_1!r_2!\cdots r_k!} \frac{x^{r_1+r_2+\cdots + r_k}}{(1!)^{r_1}(2!)^{r_2}\cdots (k!)^{r_k}}.
        \label{eq:exp_bell_polynomial2}
    \end{equation}
    
Let $\partial_{\theta}^{n} := d^{n}/d\theta^{n}$ be the differential operator. Taking \eqref{eq:bell_generating_function} and performing $\partial_{\theta}^{n}$ one gets 
    \begin{equation}
        \begin{aligned}[b]
            \partial_{\theta}^{k}\varphi_{x}(\theta) &= \sum_{n=k}^{\infty}B_{n}(x)n(n-1)(n-2)\cdots(n-k+1)\frac{\theta^{n-k}}{n!}\\
            &=\sum_{n=0}^{\infty}B_{n+k}(x)(n+k)(n+k-1)(n+k-2)\cdots(n+1)\frac{\theta^{n}}{(n+k)!}\\
            &=\sum_{n=0}^{\infty}B_{n+k}(x)\frac{\theta^{n}}{n!}.
        \end{aligned}
        \label{eq:shifted_bell_polynomial_generator}
    \end{equation}
We called \eqref{eq:shifted_bell_polynomial_generator} the $k$-shifted Bell polynomials generating function.

Considering the function $f(\theta)=e^{xe^{\theta}}$ and its Taylor series expansion
    \begin{equation*}
        e^{xe^{\theta}} = \sum_{k=0}^{\infty}\frac{(xe^{\theta})^{k}}{k!},
    \end{equation*}
and differentiating both sides of this expression with relation to $\theta$ one gets
    \begin{equation*}
    e^{xe^{\theta}}xe^{\theta} = \sum_{k=0}^{\infty}\frac{k(xe^{\theta})^{k}}{k!}.
    \end{equation*}
Taking the second derivative with relation to $\theta$ in both sides one gets
    \begin{equation*}
        e^{xe^{\theta}}[(xe^{\theta})^{2}+xe^{\theta}] = \sum_{k=0}^{\infty}\frac{k^{2}(xe^{\theta})^{k}}{k!}.
    \end{equation*}
Keeping this fashion until the n$^\text{th}$ derivative, one can write the left-hand side of the last expression as a product of the function $f(\theta)$ with the polynomial $p_n(xe^{\theta})$.
    \begin{equation}
        e^{xe^{\theta}}p_{n}(xe^{\theta}) = \sum_{k=0}^{\infty}\frac{k^{n}(xe^{\theta})^{k}}{k!}
        \label{eq:bell_differential_deduction}
    \end{equation}
Now, comparing the right-hand side of \eqref{eq:bell_differential_deduction} with \eqref{eq:dobinski_formula}, one has that $p_n(xe^{\theta}) = B_{n}(xe^{\theta})$ leading us to 
    \begin{equation}
        \partial_{\theta}^{n}e^{xe^{\theta}}=B_{n}(xe^{\theta})e^{xe^{\theta}}.
        \label{eq:iterated_exponential_identity}
    \end{equation}

Starting from \eqref{eq:bell_generating_function}, writing $\varphi_{x}(\theta) = e^{x(e^{\theta}-1)} =e^{xe^{\theta}}e^{-x}$ and using \eqref{eq:iterated_exponential_identity} we have that
    \begin{equation}
        \partial_{\theta}^{k}\varphi_{x}(\theta) = B_{k}(xe^{\theta})\varphi_{x}(\theta),\quad\text{ (see \cite{boyadzhiev2009exponential})}.
        \label{eq:diff_bell_generating}
    \end{equation}
It follows as a consequence of \eqref{eq:diff_bell_generating} and \eqref{eq:shifted_bell_polynomial_generator} that
    \begin{equation}
        B_{k}(xe^{\theta})\varphi_{x}(\theta) = \sum_{n=0}^{\infty}B_{n+k}(x)\frac{\theta^{n}}{n!}.
        \label{eq:bell_differential_identity}
    \end{equation}

At last, we recall that the Bell polynomials are sequences of binomial type. That is, $B_0(x)=1$, $B_n(0)=0$ for all $n$ and 
    \begin{equation}
        B_{n}(x+y) = \sum_{l=0}^{n}\binom{n}{l}B_{l}(x)B_{n-l}(y), \text{ for } n\geq 0.
        \label{eq:binomial_identity}
    \end{equation}
    
The existence of a connection between polynomials of binomial type and the compound Poisson process is pointed by \cite{rota1973foundations}. Lately, that connection was explored by \cite{stam1988polynomials}.

\subsection{Bell-Touchard distribution}

In this section, we take a look at some properties of the Bell-Touchard distribution covered in \cite{castellares2020two}. Although the proofs of these results are easy we include it for the sake of completeness. 

\begin{definition}
    A discrete random variable $Y$ has a Bell-Touchard distribution with parameters $(\alpha,\theta) \in \mathbb{R}_{+}^{2}$ if its probability mass function is given by 
    \begin{equation}
     \Pr[Y=x] = e^{-\alpha(e^{\theta}-1)}\frac{\theta^{x}}{x!}B_{x}(\alpha), \quad x \in \mathbb{N}_{0},
     \label{eq:bell_touchard_distribution}
 \end{equation}
 where $B_{n}(\alpha)$ is the $\text{n}^\text{th}$ single variable Bell polynomial.
\end{definition}
If $Y$ follows a Bell-Touchard distribution with parameters $(\alpha,\theta)$ then we use the notation $Y \thicksim BT(\alpha,\theta)$ in order to indicate this fact.
\begin{proposition}
    Let $Y \thicksim BT(\alpha,\theta)$. Then the probability generating function of $Y$ is given by
    \begin{equation}
        G_{Y}(s)=\exp\{\alpha(e^{s\theta}-e^{\theta})\}, \quad |s|<1.
        \label{eq:bell_probability_generating_function}
    \end{equation}
    \label{prop:probability_generating_function}
\end{proposition}
\begin{proof}
It follows from \eqref{eq:bell_touchard_distribution}, \eqref{eq:bell_generating_function} and the definition of probability generating function that
\begin{align*}
    G_{Y}(s):= \mathbb{E}[s^{y}] &= \sum_{y=0}^{\infty}s^{y}\Pr[Y=y]\\
    &= e^{-\alpha(e^{\theta}-1)}\sum_{y=0}^{\infty}\frac{B_{y}(\alpha)(s\theta)^{y}}{y!}\\
    &= e^{-\alpha(e^{\theta}-1)}e^{\alpha(e^{s\theta}-1)}\\
    &=\exp\{\alpha(e^{s\theta}-e^{\theta})\}.
\end{align*}
    
\end{proof}
\begin{proposition}
    Let the sequence of independent random variables $\{Y_i\}_{i=1}^n$ with $Y_i \thicksim BT(\alpha_{i},\theta)$, hence   ${\sum_{i=1}^{n}Y_i} \thicksim BT(\alpha, \theta)$, with $\alpha = \sum_{i=1}^{n}\alpha_{i}$.
    \label{prop: sum_of_bt_random_variables}
\end{proposition}
\begin{proof}
Taking the sequence $\{Y_i\}_{i=1}^n$, considering ${Z=\sum_{i=1}^{n}Y_i}$ and using the Proposition  \ref{prop:probability_generating_function} one has that
\begin{equation}
\begin{aligned}[b]
    G_{Z}(s) &= \prod_{i=1}^{n}G_{Y_i}(s)\\
    &=\prod_{i=1}^{n}\exp\{\alpha_{i}(e^{s\theta}-e^{\theta})\}\\
    &=\exp\{\sum_{i=1}^{n}\alpha_{i}(e^{s\theta}-e^{\theta})\}, \text{ making } \alpha = \sum_{i=1}^{n}\alpha_{i}\\
    &= \exp\{\alpha(e^{s\theta}-e^{\theta})\}.
\end{aligned}
\end{equation}
As a consequence of the uniqueness property of the probability generating function, the Bell-Touchard distribution is closed under convolution.

\end{proof}

\begin{proposition}
    If $Y \thicksim BT(\alpha,\theta)$, then the moment generating function of $Y$ is given by
    \begin{equation}
        M_{Y}(t) = \exp\{\alpha e^{\theta}[e^{\theta(e^{t}-1)}-1]\}, \quad t\in \mathbb{R}
        \label{eq:moment_generating_function}
    \end{equation}
    \label{prop:moment_generating_function}
\end{proposition}
\begin{proof}
It follows from de series expansion of Bell-Touchard probability mass function
\begin{align*}
    \mathbb{E}[e^{tY}]=M_{Y}(t) &= \sum_{y=0}^{\infty}e^{ty}\Pr[Y=y]\\
    &=e^{-\alpha(e^{\theta}-1)}\sum_{y=0}^{\infty}\frac{B_{y}(\alpha)(e^{t}\theta)^{y}}{y!}\\
    &= e^{-\alpha(e^{\theta}-1)}e^{\alpha(e^{e^{t}\theta}-1)}\\
    &=\exp\{\alpha(e^{e^{t}\theta}-e^{\theta})\}\\
    &=\exp\{\alpha e^{\theta}[e^{\theta(e^{t}-1)}-1]\}.
\end{align*}
\end{proof}
Taking the logarithm of \eqref{eq:moment_generating_function}, one gets the cumulant generating function for the Bell-Touchard distribution. Besides, considering $\mu_{1}^{'} = \mathbb{E}[Y]$, $\mu_{2}^{'} = \mathbb{E}[Y^2]$ and $\mu_{1}^{'} = \kappa_1$, $\mu_{2}^{'} = \kappa_2 + \kappa_{1}^{2}$, where $\kappa_1, \kappa_{2}$ are the first and second cumulants, we have that $\mathcal{K}(t) = \alpha e^{\theta}[e^{\theta(e^{t}-1)}-1]$ from which we get $\kappa_{1} = \alpha\theta e^{\theta}$ and $\kappa_{2} = (\theta + 1)\alpha\theta e^{\theta}$. Writing the moments as a function of the cumulants and remembering that the variance is a function of the first two moments, we get the mean and variance of the Bell-Touchard distribution; namely, $\mathbb{E}(Y) = \alpha \theta e^{\theta}$ and $\mathbb{V}\text{ar}(Y) = \theta(\theta + 1)\alpha e^{\theta}$. One can derive this result using \eqref{eq:MultiplePoisson}.

\section{Bell-Touchard counting process}

A counting process is a stochastic process $\{N(t), \,t\geq 0\}$ where $N(t) \in \mathbb{N}_{0}$ represents the number of events that occur within a time interval of length $t$. The probability of $k$ events occurring in the interval $[0,t)$ is given by ${\Pr[N(t)=k]}$, where ${\Pr[N(0)=0] = 1}$, $\Pr[N(0)=k] = 0$, for $k>0$ and ${\sum_{k \in \mathbb{N}_{0}}\Pr[N(t)=k]=1}$. Moreover, a counting process has stationary increments when the probability of $k$ events taking place in the interval  $(t_1,\,t_2)$ depends only on its length ${t=t_2-t_1}$. The process is said to possess independent increments if the number of events ocurring in disjoint time intervals are independent. Besides, a function ${f:\mathbb{R}\rightarrow\mathbb{R}}$ is said to be $o(s)$ if $\lim_{s\rightarrow 0}{f(s)}/{s}=0$. Now we are in conditions to define the Bell-Touchard process.

\begin{definition}
    A counting process $\{N(t), \,t\geq 0\}$ is said to be a Bell-Touchard process with parameters $(\alpha, \theta)\in \mathbb{R}_{+}^{2}$ if the following assumptions hold:
	\begin{enumerate}
	    \item $N(0) = 0$;
	    \item $\{N(t), \,t\geq 0\}$ has stationary and independent increments;
	    \item $ \Pr[N(t+s)-N(t) = k] = \alpha s\theta^{k}/k! + o(s)$, $k\in\mathbb{N} \text{ and } s,t>0$.
	\end{enumerate}
	\label{def:bell_processes}
    \end{definition}
The assumptions in Definition \ref{def:bell_processes} indicate that the Bell-Touchard process has non-unitary jumps with non-negligible probability. Besides, as a result of our hypotheses, it follows that:

    \begin{theorem}
     If $\{N(t), \,t\geq 0\}$ is a Bell-Touchard process with parameters $(\alpha, \theta)$ then $N(t) \thicksim BT(\alpha t,\theta)$, for all $t\geq0$.
     \label{teo:bell_touchard_theorem}
    \end{theorem}
    \begin{proof} In order to prove it we use the fact that a non-negative random variable $X$, with moment generating function $M_{X}(t)$, has Laplace transform $g(t) = M_{X}(-t) = \mathbb{E}[e^{-tX}]$, for $t\geq 0$. As a consequence of that, the Laplace transform uniquely determines the probability distribution, see \cite[Chapter~2]{ross2010introduction}.
        Let $g(t) = \mathbb{E}[\exp\{-uN(t)\}]$, for all $t$, we are going to write a differential equation for $g(t)$, noting that $g(0)=1$. Indeed, since
       \begin{equation}
           \begin{aligned}[b]
              g(t + s) &= \mathbb{E}[\exp\{-uN(t+s)\}] \\[.2cm]
                &= \mathbb{E}[\exp\{-u(N(t)+N(t+s)-N(t))\}]\\[.2cm]
                &= \mathbb{E}[\exp\{-uN(t)\}\exp\{-u(N(t+s)-N(t))\}]\\[.2cm]
                &=\mathbb{E}[\exp\{-uN(t)\}]\mathbb{E}[\exp\{-u(N(t+s)-N(t))\}]\\[.2cm]
                &=g(t)\mathbb{E}[\exp\{-uN(s)\}],
           \end{aligned}
           \label{eq:dif_laplace_1}
       \end{equation}
       where the last two steps follow from the independence and stationarity assumptions. As a result from Definition \ref{def:bell_processes}(3), one has 
       \begin{equation*}
       \begin{split}
        \Pr[N(s)=0] &=1 - \sum_{k=1}^{\infty}\Pr[N(s)=k]\\ 
                    &=1 - \sum_{k=1}^{\infty} \alpha s\frac{\theta^{k}}{k!} + o(s)\\
                    &=1 - \alpha s(e^{\theta}-1)+o(s).
        \end{split}
       \end{equation*}
       Therefore, conditioning in $N(s) = k,\text{ for all } k \in \mathbb{N}_{0} $, one gets
       \begin{equation}
           \begin{aligned}[b]
              \mathbb{E}[\exp\{-u(N(s)\}] &= \sum_{k=0}^{\infty}e^{-uk}\Pr[N(s) = k]\\[.2cm]
              &= 1 - \alpha s(e^{\theta}-1)+o(s) + \sum_{k=1}^{\infty}e^{-uk}\alpha s\frac{\theta^{k}}{k!}\\[.2cm]
              &= 1 - \alpha s(e^{\theta}-1)+o(s) + \alpha s[e^{\theta e^{-u}}-1]\\[.2cm]
              &= 1 + \alpha s[e^{\theta e^{-u}}-e^{\theta}] + o(s).
           \end{aligned}
           \label{eq:dif_laplace_2}
       \end{equation}
       And from $\eqref{eq:dif_laplace_1}$ and $\eqref{eq:dif_laplace_2}$, it follows that
       \begin{equation*}
           g(t+s) = g(t)(1 + \alpha s[e^{\theta e^{-u}}-e^{\theta}]) + o(s), 
       \end{equation*}
       thus
       \begin{equation*}
           \frac{g(t+s)-g(t)}{s} = g(t) \alpha [e^{\theta e^{-u}}-e^{\theta}] + \frac{o(s)}{s}
       \end{equation*}
       and taking $s\rightarrow 0$,
       \begin{equation*}
           \frac{dg(t)}{dt} = g(t)\alpha [e^{\theta e^{-u}}-e^{\theta}].
       \end{equation*}
    Solving the differential equation, one gets
       \begin{equation*}
           g(t) = \exp\{\alpha te^{\theta}[e^{\theta (e^{-u}-1)}-1]\},
       \end{equation*}
       which is the Laplace transform of a random variable $X \thicksim BT(\alpha t, \theta)$, or the moment generating function $\eqref{eq:moment_generating_function}$ evaluated in $-t$.
       
    \end{proof}
    As a result of the Theorem \ref{teo:bell_touchard_theorem}, we have another definition for the Bell-Touchard process.
  \begin{definition}
    A counting process $\{N(t), \,t\geq 0\}$ is called a Bell-Touchard process with parameters $(\alpha, \theta)$ if the following assumptions hold:
	\begin{enumerate}
	    \item $N(0) = 0$;
	    \item $\{N(t), \,t\geq 0\}$ has stationary and independent increments;
	    \item For all $k \in \mathbb{N}_{0}$, and $t>0$:
	    $$ \Pr[N(t)=k] = e^{-\alpha t(e^{\theta}-1)}\frac{\theta^{k}}{k!}B_{k}(\alpha t),$$ 	where $B_k(\alpha)$ is the $k^{\text{th}}$ Bell polynomial.

	\end{enumerate}
	\label{def:bell_processes2}
    \end{definition}
    
\subsection{Properties}

Consider a multiple Poisson process $\{\xi(t), \,t\geq 0\}$ depending on the non-negative sequence ${\{c_{n}\}_{n\geq1}}$ with $\sum c_n<\infty$. That is, $\xi(0)=0$, the process has stationary and independent increments, and for any $k\in \mathbb{N}$ and $t>0$:

  \begin{equation}
	        \Pr[\xi(t)=k] =\exp{\left \{-\sum_{n=1}^{\infty}tc_n \right \}} \sum_{r_1+2r_2+\dots+kr_k=k}\frac{(tc_1)^{r_1}(tc_2)^{r_2} \dots (tc_k)^{r_k}}{r_1!r_2!\dots r_k!}.
	        \label{eq:poisson_composed_family}
	\end{equation}

\bigskip
The following result states that a Bell-Touchard process is a multiple Poisson process too.

	\begin{proposition}
	    If $\{N(t), t\geq0\}$ is a Bell-Touchard process with parameters  $(\alpha, \theta)$, then $N(t)$ has mass probability function given by \eqref{eq:poisson_composed_family}, where $c_n = \alpha\theta^n/n!$.
	    \label{prop:poisson_composed_family_member}
	\end{proposition}
	\begin{proof}
	    From Theorem \ref{teo:bell_touchard_theorem}, taking \eqref{eq:bell_touchard_distribution}, replacing $B_{x}(\alpha t)$ by the right hand side of \eqref{eq:exp_bell_polynomial2} and making $x = k$ we have
        \begin{equation*}
             \Pr[N(t)=k] =\theta^{k}e^{-\alpha t(e^\theta-1)}\sum_{r_1+2r_2+\dots+kr_k=k}\frac{(\alpha t)^{r_1+r_2+\dots r_k}}{(1!)^{r_1}(2!)^{r_2}\dots (k!)^{r_k}r_1!r_2!\dots r_k!},
        \end{equation*}
        however $k=r_1+2r_2+\dots+kr_k$ such that $\theta^k = \theta^{r_1+2r_2+\dots+kr_k}$, therefore
        \begin{equation*}
            \Pr[N(t)=k] = e^{-\alpha t(e^\theta-1)} \sum_{r_1+2r_2+\dots+kr_k=k}\frac{\theta^{r_1+2r_2+\dots +kr_k}(\alpha t)^{r_1+r_2+\dots r_k}}{(1!)^{r_1}(2!)^{r_2}\dots (k!)^{r_k}r_1!r_2!\dots r_k!}.
        \end{equation*}
        Expanding the exponential function in the last expression using Taylor Series centered at zero and combining $\theta 's \textrm{ and } \alpha t's$ under the powers $r's$ we get
        \begin{align*}
		  \Pr[N(t)=k] &=\exp{\left \{-\sum_{n=1}^{\infty}\alpha t\frac{\theta^n}{n!} \right \}} \sum_{r_1+2r_2+\dots+kr_k=k}\frac{(t\alpha\theta)^{r_1}(t\alpha \theta^2)^{r_2} \dots (t\alpha\theta^k)^{r_k}}{(1!)^{r_1}(2!)^{r_2}\dots (k!)^{r_k}r_1!r_2!\dots r_k!}.
	    \end{align*}
	    Comparing the last expression with \eqref{eq:poisson_composed_family}, we observe that $c_n = \alpha\theta^n/n!$, such that the convergent sequence is given by $\{c_n = \alpha\theta^n/n! \}_{n \geq 1}$, closing the demonstration.
	
	\end{proof}
	
	Let $\tau_{0} = 0$ and define a random variable $\tau_i := \inf\{t>0 : N(t)-N(\tau_{i-1})\geq1\}$ for all $i\in \mathbb{N}$.
	
    \begin{proposition}
        $\tau_1, \tau_2, \dots$ are independent and identically distributed random variables, following an exponential distribution with $\lambda = \alpha(e^{\theta}-1)$.
        \label{prop:waiting_times_distribution}
    \end{proposition}
    \begin{proof}
        It follows from \ref{teo:bell_touchard_theorem} that $\tau_{1} \thicksim \text{Exp}(\alpha (e^{\theta}-1))$. Indeed,
        \begin{equation*}
            \Pr[\tau_{1}>t] = \Pr[N(t)=0] = e^{-\alpha t(e^{\theta}-1)}.
        \end{equation*}
        For $\tau_{2}$, we have that
        \begin{align*}
            {\Pr[\tau_2>s\bigm\lvert\tau_1=t]} &= \Pr[\textrm{0 events in } (t,t+s]\bigm\lvert\tau_1=t],\\
            &= \Pr[\textrm{0 events in } (t,t+s]], \textrm{ (by independent increments)}\\
            &=e^{-\alpha s(e^{\theta}-1)} \textrm{ (by stationary increments)},
        \end{align*}
        and by applying induction for $i \geq 2$, we complete the proof.
        
    \end{proof}
    
    Define $\delta_n := \sum_{i=0}^{n}\tau_i$. Therefore, $\delta_n$ is the length of the time interval within events of the process take place.
    Note that, according to the Definition  \ref{def:bell_processes}, $N(\delta_n)\geq n$. In other words, the number of events within a specific time interval will be equal or greater than the number of jumps of the process within the same time interval, where $N(\delta_n)=n$ when all the jumps heights are unitary. Combining it with the previous result, we conclude that the Bell-Touchard process is also a compound Poisson process.
    \begin{proposition}
         Let $\{M(t), t\geq 0\}$ be a Poisson process with rate  $\alpha  (e^{\theta}-1)$ and $\{X_i\}_{i\geq1}$ a sequence of independent and identically distributed random variables  which also is independent of $\{M(t), t\geq 0\}$ with $X_i$ having probability mass function given by $$\Pr[X_{i}=x] = \frac{\theta^x/x!}{e^{\theta}-1}, \text{ for all } x \in \mathbb{N}.$$ Then, the following equality in distribution holds.
	    \begin{equation}
	        N(t) \stackrel{d}{=} \sum_{i=1}^{M(t)}X_i, \text{ for } t\geq 0.
	        \label{eq:bell_compound_poisson_family_member}
	    \end{equation}
	    \label{prop: bell_compound_poisson_family_member}
	    Where $N(t)$ is the number of events of the Bell-Touchard process.
    \end{proposition}
    \begin{proof}
        Let $\gamma(s)=\mathbb{E}[s^{X_{i}}]$ be a probability generating function of $X_i$, thus
        \begin{equation}
            \gamma(s) = \frac{e^{\theta s}-1}{e^{\theta}-1}.
            \label{eq: pgf_zero_truncated_poisson}
        \end{equation}
        Considering \eqref{eq:bell_compound_poisson_family_member},
        \begin{equation}
            \begin{aligned}[b]
                \mathbb{E}[s^{N(t)}] &= \mathbb{E}\{\mathbb{E}[s^{\sum_{i=1}^{M(t)}X_i}|M(t)]\}\\
                &=\sum_{i=0}^{\infty}\frac{e^{-\alpha t(e^{\theta}-1)}[\alpha t(e^{\theta}-1)]^{i}}{i!}\gamma(s)^{i}\\
                &=\exp\{-\alpha t(e^{\theta}-1)(1-\gamma(s))\}
            \end{aligned}
            \label{eq:deduction_pgf_compound}
        \end{equation}
        and replacing \eqref{eq: pgf_zero_truncated_poisson} on \eqref{eq:deduction_pgf_compound} we have
        \begin{equation}
            G_{N(t)}(s):= \mathbb{E}[s^{N(t)}] = \exp\{\alpha t[e^{\theta s}-e^{\theta}]\} \text{ , }|s|<1,
            \label{eq:compound_process_pgf}
        \end{equation}
        which is the probability generating function of a Bell-Touchard random variable with parameters  $\alpha t \text{ e } \theta$, according to the Proposition \ref{prop:probability_generating_function}.
        
    \end{proof}
    The following result is about the superposition of Bell-Touchard processes, which result in a Bell-Touchard process since all random variables involved have the same parameter $\theta$.

    \begin{proposition}
        Let $\{N_{i}(t), t\geq 0\}_{i \in \mathbb{N}}$ be a family of Bell-Touchard process with parameters $\{\alpha_i,\theta\}_{i \in \mathbb{N}}$ and $\tilde{N}_{n}(t) = \sum_{i=1}^{n}N_{i}(t)$, for all $t\geq0$. Thus $\{\tilde{N}_{n}(t), t\geq 0\}$ is a Bell-Touchard process with parameters $(\tilde{\alpha}_{n}, \theta)$, where $\tilde{\alpha}_{n}:=\sum_{i=1}^{n}\alpha_{i}$.
        \label{prop:superposition_same_theta}
    \end{proposition}
    \begin{proof}
        For $n$, we assume that the probability mass function of $\tilde{N}_{n}(t)$ is given by:
        \begin{equation*}
            \Pr[\tilde{N}_{n}(t)=k]=e^{-\tilde{\alpha}_{n}t(e^{\theta}-1)}\frac{\theta^{k}}{k!}B_{k}(\tilde{\alpha}_{n}t)
        \end{equation*}
        For $n=2$, we have
        \begin{align*}
            \Pr[N_1(t)+N_2(t)=k] &= \sum_{i=0}^{k}\Pr[N_1(t)=i]\Pr[N_2(t)=k-i]\\
            &=\sum_{i=0}^{k}e^{-\alpha_1 t(e^{\theta}-1)}\frac{\theta^{i}}{i!}B_{i}(\alpha_1 t)e^{-\alpha_2 t(e^{\theta}-1)}\frac{\theta^{k-i}}{(k-i)!}B_{k-i}(\alpha_2 t)\\
            &=e^{-(\alpha_1 t+\alpha_2 t)(e^{\theta}-1)}\frac{\theta^{k}}{k!}\sum_{i=0}^{k}\binom{k}{i}B_{i}(\alpha_1 t)B_{k-i}(\alpha_2 t)\\
            &=e^{-(\alpha_1+\alpha_2)t(e^{\theta}-1)}\frac{\theta^{k}}{k!}B_{k}((\alpha_1+\alpha_2)t)\\
            \Pr[\tilde{N}_{2}(t)=k]&=e^{-\tilde{\alpha}_{2}t(e^{\theta}-1)}\frac{\theta^{k}}{k!}B_{k}(\tilde{\alpha}_{2}t)
        \end{align*}
        We assume the induction hypothesis for $n$. Therefore, considering $n+1$, we have
        \begin{align*}
            \Pr[\tilde{N}_{n}(t)+N_{n+1}(t)=k] &= \sum_{i=0}^{k}\Pr[\tilde{N}_{n}(t)=i]\Pr[N_{n+1}(t)=k-i]\\
            &=\sum_{i=0}^{k}e^{-\tilde{\alpha}_{n}t(e^{\theta}-1)}\frac{\theta^{i}}{i!}B_{i}(\tilde{\alpha}_{n}t)e^{-\alpha_{n+1} t(e^{\theta}-1)}\frac{\theta^{k-i}}{(k-i)!}B_{k-i}(\alpha_{n+1} t)\\
            &=e^{-(\tilde{\alpha}_n t+\alpha_{n+1} t)(e^{\theta}-1)}\frac{\theta^{k}}{k!}\sum_{i=0}^{k}\binom{k}{i}B_{i}(\tilde{\alpha}_n t)B_{k-i}(\alpha_{n+1} t)\\
            &=e^{-(\tilde{\alpha}_n+\alpha_{n+1})t(e^{\theta}-1)}\frac{\theta^{k}}{k!}B_{k}((\tilde{\alpha}_n+\alpha_{n+1})t)\\
            \Pr[\tilde{N}_{n+1}(t)=k]&=e^{-\tilde{\alpha}_{n+1}t(e^{\theta}-1)}\frac{\theta^{k}}{k!}B_{k}(\tilde{\alpha}_{n+1}t),
        \end{align*}
        concluding the demonstration by induction.
        
    \end{proof}
    Suppose we have two Bell-Touchard process $\{N_{1}(t), t\geq 0\}$ and $\{N_{2}(t), t\geq 0\}$, with respective parameters $(\alpha_{1}, \theta_{1}) \text{ and } (\alpha_{2}, \theta_{2})$ and we would like to generate a process $\{V(t), t\geq 0\}$ where $V(t) = N_{1}(t)+N_{2}(t)$. In this situation, the resulting process is not a Bell-Touchard process, however it can be easily deduced from the properties of the compound Poisson family, in this case, the fact that this family is closed for convolutions (see \cite[chapter~5]{ross2010introduction}).
    \begin{proposition}
        Let $\{N_{1}(t), t\geq 0\}$ and $\{N_{2}(t), t\geq 0\}$ be Bell-Touchard processes with parameters $(\alpha_{1}, \theta_{1}) \text{ and } (\alpha_{2}, \theta_{2})$ and $\nu(\theta_{i}):=e^{\theta_{i}}-1$. Thus $\{V(t), t\geq 0\}$ where $V(t) = N_{1}(t)+N_{2}(t)$ is a compound Poisson process.
        \begin{equation*}
            V(t) = \sum_{i=1}^{J(t)}X_{i} \text{ for } t\geq 0,
        \end{equation*}
        where $\{J(t),t\geq0\}$ is a Poisson process with rate $\lambda = \alpha_{1}\nu(\theta_{1})+\alpha_{2}\nu(\theta_{2})$ and the probability mass function of $X_i$ is given by:
        \begin{equation*}
            \Pr[X_{i}=x] = \frac{\theta_{1}^{x}+\theta_{2}^{x}}{[\nu(\theta_{1})+\nu(\theta_{2})]x!}, \text{ } x \in \mathbb{N}_{0}.
        \end{equation*}
        \label{prop: sum_different_theta}
    \end{proposition}
    \begin{proof}
        Let $\lambda_{i} = \alpha_{i}t(e^{\theta_{i}}-1)$. Replacing $\lambda_{i}$ in \eqref{eq:deduction_pgf_compound} we see that
        \begin{align*}
            G_{V(t)}(s)&=\exp\{\lambda_{1}(\gamma_{1}(s)-1)+\lambda_{2}(\gamma_{2}(s)-1)\}\\
            &=\exp\{\lambda_{1}\gamma_{1}(s)+\lambda_{2}\gamma_{2}(s)-\lambda_{1}-\lambda_{2})\}\\
            &=\exp\{-\lambda[1-\gamma(s)]\}
        \end{align*}
        where $\gamma(s) = (\lambda_{1}\gamma_{1}(s)+\lambda_{2}\gamma_{2}(s))/\lambda$ and  $\lambda = \lambda_{1} + \lambda_{2}$.
        
    \end{proof}
    
        Another interesting result about the Bell-Touchard process appears when one thins it. Consider a Bell-Touchard process $\{N(t),t\geq0\}$ with parameters $(\alpha, \theta)$. Suppose that each time a jump occurs, one classifies it as either a type I or a type II. Besides, suppose that $p:=\Pr[\text{type I}]=1-\Pr[\text{type II}]$. Let $N_{1} := N_1(t)$ denote the number of events of type I and  $N_{2} := N_2(t)$ the number of events of type II. We say that $N_{1} \text{ and } N_{2}$ are decompositions of the original process. 
    \begin{theorem}
        Let ${\{N(t), t\geq0\}}$ be a Bell-Touchard process with parameters $(\alpha,\theta)$. If $\{N_{1}(t),t\geq0\}$ is a decomposition, with rate $p$, of ${\{N(t), t\geq0\}}$, then $N_{1}(t)$ has probability mass function given by:
        \begin{equation}
            \Pr[N_{1}(t)=k] = \frac{(p\theta)^{k}}{k!}\exp\{-\alpha t \widetilde{\theta}(e^{p\theta}-1)\}B_{k}(\alpha t\widetilde{\theta}),
        \end{equation}
        where $\widetilde{\theta}:=e^{(1-p)\theta}$.
        \label{teo:bell_thinning}
    \end{theorem}
    \begin{proof}
        Note that the conditional distribution of $N_{1}$ given $N=n$ is binomial with parameters $n,p$. That is
        \begin{equation*}
            \Pr[N_{1}=k | N=n] = \binom{n}{k}p^{k}(1-p)^{n-k}.
        \end{equation*}
        Conditioning in $N$,
        \begin{align*}
            \Pr[N_{1}=k] &= \sum_{n=k}^{\infty}\Pr[N_{1}=k | N=n]\Pr[N=n]\\
            &=  \sum_{n=k}^{\infty}\binom{n}{k}p^{k}(1-p)^{n-k}\frac{\theta^{n}}{n!}e^{-\alpha t(e^{\theta}-1)}B_{n}(\alpha t)\\
            &= e^{-\alpha t(e^{\theta}-1)}\frac{(p\theta)^{k}}{k!}\sum_{n=k}^{\infty}(1-p)^{n-k}\frac{\theta^{n-k}}{(n-k)!}B_{n}(\alpha t)\\
            &= e^{-\alpha t(e^{\theta}-1)}\frac{(p\theta)^{k}}{k!}\sum_{n=0}^{\infty}B_{n+k}(\alpha t)\frac{((1-p)\theta)^{n}}{n!},
        \end{align*}
        but from \eqref{eq:bell_differential_identity}
        \begin{equation*}
          \Pr[N_{1}=k] =  e^{-\alpha t(e^{\theta}-1)}\frac{(p\theta)^{k}}{k!}B_{k}(\alpha te^{(1-p)\theta})\varphi_{\alpha t}((1-p)\theta),
        \end{equation*}
        and considering \eqref{eq:bell_generating_function}, one has that $\varphi_{\alpha t}((1-p)\theta) = e^{\alpha t[e^{(1-p)\theta}-1]}$, hence
        \begin{align*}
            \Pr[N_{1}=k] &=  e^{-\alpha t(e^{\theta}-1)}\frac{(p\theta)^{k}}{k!}B_{k}(\alpha te^{(1-p)\theta})e^{\alpha t[e^{(1-p)\theta}-1]}\\
            &=\frac{(p\theta)^{k}}{k!}e^{-\alpha te^{(1-p)\theta}(e^{p\theta}-1)}B_{k}(\alpha te^{(1-p)\theta})
        \end{align*}
        and we finish the proof setting $\widetilde{\theta}:=e^{(1-p)\theta}$.
        
    \end{proof}
    A consequence of the Theorem \ref{teo:bell_thinning} is that $\{N_{1}(t), t\geq0\}$ and $\{N_{2}(t), t\geq0\}$ are Bell-Touchard process. However, they are not independent process and one can check this finding the joint probability mass function of  $(N_{1}(t),N_{2}(t))$. Consider 
    \begin{equation*}
        N(t) = N_{1}(t)+N_{2}(t).
    \end{equation*}
    And let $n$, $k$ be positive integers. Note that
    \begin{equation*}
        \Pr[N_{1}(t)=k,N_{2}(t)=n] = \Pr[N_{1}(t)=k,N_{2}(t)=n \vert N(t)=n+k]\Pr[N(t)=n+k].
    \end{equation*}
    Given that $N(t) = n+k$, considering that events of type I or II occur independently of any other events, the number of events of type I follow a binomial distribution parameters $n+k \text{ and } p$. Thus
    \begin{equation*}
        \Pr[N_{1}(t)=k,N_{2}(t)=n \vert N(t)=n+k] = \binom{n+k}{k}p^{k}(1-p)^{n}.
    \end{equation*}
    Combining these results,
    \begin{align}
        \Pr[N_{1}(t)=k,N_{2}(t)=n] &= \binom{n+k}{k}p^{k}(1-p)^{n}e^{-\alpha t(e^{\theta}-1)}\frac{(\alpha t)^{n+k}}{(n+k)!}B_{n+k}(\alpha t).
        \label{eq:join_bell}
    \end{align}
    However, one can not write $B_{n+k}(\alpha t)$ as an independent product of Bell polynomials $B_{n}(\alpha t)B_{k}(\alpha t)$, therefore one can not have $\Pr[N_{1}(t)=k,N_{2}(t)=n] = \Pr[N_{1}(t)=k]\Pr[N_{2}(t)=n]$, implying that $N_{1}(t) \text{ and } N_{2}(t)$ are not independent.
    \begin{remark}
    Another strategy for proving the lack of independence is using the moment generating function of \eqref{eq:join_bell}.
    \begin{equation}
        M_{N_{1},N_{2}}(t_1,t_2) = \exp\{\alpha t[e^{\theta(pe^{t_2}+qe^{t_1})}-e^{\theta}]\}, \text{ where } q = 1-p.
        \label{eq:join_mgf}
    \end{equation}
    One can check that
    \begin{align*}
        M_{N_{1}}(t_1) = M_{N_{1},N_{2}}(t_1,0) &= \exp\{\alpha t[e^{\theta(p+qe^{t_1})}-e^{\theta}]\} \\
        M_{N_{2}}(t_2) = M_{N_{1},N_{2}}(0,t_2) &= \exp\{\alpha t[e^{\theta(pe^{t_2}+q)}-e^{\theta}]\}.
    \end{align*}
    Now one can use the fact that $N_1$ e $N_2$ are independent if, and only if, $M_{N_{1},N_{2}}(t_1,t_2)=M_{N_{1}}(t_1)M_{N_{2}}(t_2)$, which is not the case. (see \cite[chapter~8]{schinazi2011probability}).
    \end{remark}
    
    In orther to show that $N_{1}(t)$ is a Bell-Touchard process we have to prove that $\{N_{1}(t), t\geq0\}$ has stationay and independent increments. Therefrom, for $s < t$ and any $k \in \mathbb{N}$
    \begin{equation*}
        \Pr[N_{1}(t)-N_{1}(s)=k] = \sum_{n\geq0}\Pr[N_{1}(t)-N_{1}(s)=k,N(t)-N(s)=n+k].
    \end{equation*}
    Conditioning in $N(s)-N(t)=n+k$, the distribution of $N_{1}(t)-N_{1}(s)$ is binomial with parameters $n+k$ e $p$. Thus
    \begin{align*}
         \Pr[N_{1}(t)-N_{1}(s)=k] &= \sum_{n\geq0}\binom{n+k}{k}p^{k}(1-p)^{n}\Pr[N(t)-N(s)=n+k],
    \end{align*}
    but as $N(s)-N(t)\thicksim BT(\alpha(t-s),\theta)$, it follows that
    \begin{align*}
         \Pr[N_{1}(t)-N_{1}(s)=k] &= \sum_{n\geq0}\binom{n+k}{k}p^{k}(1-p)^{n}\exp[-\alpha(t-s)(e^{\theta}-1)]\frac{\theta^{n+k}}{(n+k)!}B_{n+k}(\alpha(t-s)),\\
         &=\exp[-\alpha(t-s)\widetilde{\theta}(e^{p\theta}-1)]\frac{(p\theta)^{k}}{k!}B_{k}(\alpha(t-s)\widetilde{\theta}),
    \end{align*}
    where the result follows from the same strategy applied in the proof of Theorem \ref{teo:bell_thinning}, setting $\widetilde{\theta} :=e^{(1-p)\theta}$. This shows that the distribution of  $N_{1}(t)-N_{1}(s)$ depends only on the length of the interval $(s,t]$. We summarize this in a corollary.
    \begin{corollary}
        $\{N_{1}(t),t\geq0\}$ is a Bell-Touchard process with parameters $(\alpha^{*},\theta^{*})$. Where $\alpha^{*}:=\alpha e^{(1-p)\theta} \text{ and } \theta^{*}:=p\theta$, with $p\in(0,1)$.
    \end{corollary}
    
    The Theorem \ref{teo:bell_thinning} can be generalized for any finite number of decompositions.
    
    \begin{corollary}
        Let $\{N_{i}(t), t\geq0\}$ be the $i^{\text{th}}$ decomposition of a Bell-Touchard process $\{N(t), t\geq0\}$ with parameters $(\alpha, \theta)$, such that the proportion of events of the $i^{\text{th}}$ decomposition is denoted by $p_i:=\Pr[i^{\text{th}} \text{ type event}]$, where $i\in \{1,\dots,n\}$ and $\sum_{i=1}^{n}p_{i}=1$. Therefore $\{N_{i}(t),t\geq0\}$ is a Bell-Touchard process with parameters $\alpha^{*}_{i}:=\alpha e^{(1-p_{i})\theta} \text{ and } \theta^{*}_{i}:=p_{i}\theta$.
    \end{corollary}
    
    \begin{definition}
        Consider two independent homogeneous Poisson process $\{N_{1}(t),t\geq0\}$ with parameter $\nu > 0$ and  $\{N_{2}(t),t\geq0\}$ with parameter $\omega>0$. The process  $\{\Hat{N}(t),t>0\}$ where $\Hat{N}(t) :=N_{1}(N_{2}(t))\text{, for } t>0$ is called iterated Poisson process.
    \end{definition}
     The process $\{\Hat{N}(t),t>0\}$ is generated composing $\{N_{1}(t),t\geq0\}$ and $\{N_{2}(t),t\geq0\}$. Furthermore, its path takes place on the path of the inside process. One can check more about these results in \cite{orsingher2012compositions,di2015compound}.
     
    \begin{theorem}
         The iterated Poisson process $\{\Hat{N}(t),t>0\}$ is a Bell-Touchard process with parameters ${\alpha:=\omega e^{-\nu}}$ and $ \theta:=\nu$.
         \label{teo:poisson_composition}
    \end{theorem}
    \begin{proof}
        Note that 
        \begin{equation*}
            \Pr[\Hat{N}(t)=k]=\sum_{r=0}^{\infty}\Pr[N_{1}(r)=k\lvert N_{2}(t)=r]\Pr[N_{2}(t)=r],
        \end{equation*}
        but as the process are independent
        \begin{align*}
            \Pr[\Hat{N}(t)=k]&=\sum_{r=0}^{\infty}\Pr[N_{1}(r)=k]\Pr[N_{2}(t)=r],\\
                            &=\sum_{r=0}^{\infty}\frac{e^{-\nu r}(\nu r)^{k}}{k!}\frac{e^{-\omega t}(\omega t)^{r}}{r!},\\
                            &=\frac{e^{-\omega t}\nu^{k}}{k!}\sum_{r=0}^{\infty}\frac{r^{k}(\omega te^{-\nu})^{r}}{r!},
        \end{align*}
        multiplying the last expression by $e^{-\omega te^{-\nu}}/e^{-\omega te^{-\nu}}$ and using \eqref{eq:dobinski_formula}, we have
        \begin{equation*}
             \Pr[\Hat{N}(t)=k] = e^{-\omega te^{-\nu}[e^{\nu}-1]}\frac{\nu^{k}}{k!}B_{k}(\omega te^{-\nu}), \qquad k \geq0,t>0.
        \end{equation*}
        The last expression corresponds to the Bell-Touchard probability mass function with $\alpha:=\omega e^{-\nu}\text{ and } \theta:=\nu$.
    \end{proof}

\subsection{Generalizations}

\begin{definition}
    A counting process $\{N(t), \,t\geq 0\}$ is a nonhomogeneous Bell-Touchard process with parameters $(\alpha(t), \theta)$, where $\alpha(t):\mathbb{R}\rightarrow\mathbb{R}_{+}$, if the following assumptions hold:
	\begin{enumerate}
	    \item $N(0) = 0$;
	    \item $\{N(t), \,t\geq 0\}$ has independent increments;
	    \item $ \Pr[N(t+s)-N(t) = k] = \alpha(t) s{\theta^{k}}/{k!} + o(s).$
	\end{enumerate}
	\label{def:nonhomogeneous_bell_processes}
    \end{definition}
    
    Note that we fix $\theta$ and define $\alpha(t)$ as a function of time. We also define the average jumps function as follows.
    \begin{equation*}
        m(t) := \int_{0}^{t}\alpha(w)dw.
    \end{equation*}
    
    \begin{lemma}
    Let $\{N(t), t\geq0\}$ be a nonhomogeneous Bell-Touchard process and $N_{t}(s)=N(t+s)-N(t)$ with $s\geq0$. Thus $\{N_{t}(s), t\geq0\}$ is a nonhomogeneous Bell-Touchard process as well, where $\alpha_{t}(s)=\alpha(t+s)$.
        \label{lemma: nonhomogeneous_shifted_bell_touchard}
    \end{lemma}
    The average jumps function of $N_{t}(s)$ is given by
    \begin{align*}
        m_{t}(s) &= \int_{0}^{s}\alpha_{t}(w)dw\\
        &=\int_{0}^{s}\alpha(t+w)dw\\
        &=\int_{t}^{t+s}\alpha(u)du\\
        m_{t}(s) &= m(t+s)-m(s).
    \end{align*}
    Considering that, we present the next result.
    \begin{theorem}
        Let $\{N(t), t\geq0\}$ be a nonhomogeneous Bell-Touchard process with parameters $(\alpha(t), \theta)$. Then $N(t)$ has probability mass function given by
        \begin{equation}
            \Pr[N(t)=k]=\exp\{-m(t)(e^{\theta}-1)\}\frac{\theta^{k}}{k!}B_{k}(m(t)).
        \end{equation}
        \label{teo:nonhomogeneos_bell_touchard}
    \end{theorem}
    
    \begin{proof}
        Following the same fashion used in the proof of Theorem \ref{teo:bell_touchard_theorem}, writing
        \begin{equation*}
            g(t) = \mathbb{E}[\exp\{-uN(t)\}].
        \end{equation*}
        We need to find and solve a differential equation for $g(t)$.
       \begin{equation}
           \begin{aligned}[b]
              g(t + s) &= \mathbb{E}[\exp\{-uN(t+s)\}] \\
                &= \mathbb{E}[\exp\{-u(N(t)+N(t+s)-N(t))\}]\\
                &= \mathbb{E}[\exp\{-uN(t)\}\exp\{-u(N(t+s)-N(t))\}]\\
                &=g(t)\mathbb{E}[\exp\{-uN_{t}(s)\}],
           \end{aligned}
           \label{eq:nonhomogeneous_dif_laplace_1}
       \end{equation}
       where the last step is due to the independent increments assumption. AS a consequence of assumption (3) from the Definition \ref{def:nonhomogeneous_bell_processes}, we have 
       \begin{equation*}
        \Pr[N_{t}(s)=0] = 1 - \alpha(s+t)s(e^{\theta}-1)+o(s).
       \end{equation*}
       Conditioning in $N_{t}(s) = k,\forall k \in \mathbb{N}_{0} $ and using the Lemma \ref{lemma: nonhomogeneous_shifted_bell_touchard}       \begin{equation}
           \begin{aligned}[b]
              \mathbb{E}[\exp\{-uN_{t}(s)\}] &= \sum_{k=0}^{\infty}e^{-uk}\Pr[N_{t}(s) = k]\\
              &= 1 - \alpha(t+s) s(e^{\theta}-1)+o(s) + \sum_{k=1}^{\infty}e^{-uk}\alpha(t+s)s\frac{\theta^{k}}{k!}\\
              &= 1 + \alpha(t+s)s[e^{\theta e^{-u}}-e^{\theta}] + o(s)
           \end{aligned}
           \label{eq:nonhomogeneous_dif_laplace_2}
       \end{equation}
       Thus, considering $\eqref{eq:nonhomogeneous_dif_laplace_1}$ and $\eqref{eq:nonhomogeneous_dif_laplace_2}$, we have that
       \begin{equation*}
           g(t+s) = g(t)(1 + \alpha(t+s)s[e^{\theta e^{-u}}-e^{\theta}]) + o(s), 
       \end{equation*}
       hence
       \begin{equation*}
           \frac{g(t+s)-g(t)}{s} = g(t) \alpha(t+s)[e^{\theta e^{-u}}-e^{\theta}] + \frac{o(s)}{s}
       \end{equation*}
       and taking $s\rightarrow 0$, we get
       \begin{equation*}
           \frac{dg(t)}{dt} = g(t)\alpha(t)[e^{\theta e^{-u}}-e^{\theta}].
       \end{equation*}
       Solving the differential equation with initial condition $g(0)=1$ we have that
       \begin{equation*}
           g(t) = \exp\{m(t)e^{\theta}[e^{\theta (e^{-u}-1)}-1]\}
       \end{equation*}
       which corresponds to the Laplace transform of a Bell-Touchard random variable $X \thicksim BT(m(t), \theta)$, where $\alpha \text{ is replaced by } m(t)$.
       
    \end{proof}
    A consequence of Theorem \ref{teo:nonhomogeneos_bell_touchard} and Proposition \ref{prop: bell_compound_poisson_family_member} is that $\{N(t), t\geq0\}$ is also a nonhomogeneous Poisson process. One can check this using probability generating functions. 
    
    \begin{theorem}
         Consider the process $\{\Hat{\mathcal{N}}(t),t>0\}$ with $\Hat{\mathcal{N}}(t) = N_{1}(\mathcal{N}_{2}(t))$ where $\{N_{1}(t),t\geq0\}$ is a Poisson process with parameter $\nu>0$ and $\{\mathcal{N}_{2}(t),t\geq0\}$ is a nonhomogeneous Poisson process with parameter $\lambda(t),t>0$ (both independents). Therefore, $\{\Hat{\mathcal{N}}(t),t>0\}$ is a nonhomogeneous Bell-Touchard process with parameters $\alpha:=m(t) e^{-\nu}\text{ and } \theta:=\nu$.
         \label{teo:poisson_composition_nonhomogeneous}
    \end{theorem}
    \begin{proof}
        The proof follows the same fashion we used proving Theorem \ref{teo:poisson_composition}.
        Setting the following equality
        \begin{equation*}
            \Pr[\Hat{\mathcal{N}}(t)=k]=\sum_{r=0}^{\infty}\Pr[N_{1}(r)=k\lvert \mathcal{N}_{2}(t)=r]\Pr[\mathcal{N}_{2}(t)=r],
        \end{equation*}
        and by independence of the processes
        \begin{align*}
            \Pr[\Hat{\mathcal{N}}(t)=k]&=\sum_{r=0}^{\infty}\Pr[N_{1}(r)=k]\Pr[\mathcal{N}_{2}(t)=r],\\
                            &=\sum_{r=0}^{\infty}\frac{e^{-\nu r}(\nu r)^{k}}{k!}\frac{e^{-m(t)}m(t)^{r}}{r!},\\
                            &=\frac{e^{-m(t)}\nu^{k}}{k!}\sum_{r=0}^{\infty}\frac{e^{-\nu r}r^km(t)^r}{r!},
        \end{align*}
        multiplying the last expression by $e^{-m(t)e^{-\nu}}/e^{-m(t)e^{-\nu}}$ and using \eqref{eq:dobinski_formula}, we have that
        \begin{equation*}
             \Pr[\Hat{\mathcal{N}}(t)=k] = e^{-m(t)e^{-\nu}[e^{\nu}-1]}\frac{\nu^{k}}{k!}B_{k}(m(t)e^{-\nu}), \qquad k \geq0,t>0,
        \end{equation*}
        where $m(t) := \int_{0}^{t}\lambda(s)ds$. Comparing the last expression with Theorem \ref{teo:nonhomogeneos_bell_touchard} finishes the proof.
        
    \end{proof}
\begin{definition}
Let $\{N(t), t \geq 0\}$ be a Bell-Touchard process and let $\alpha$ be a positive random variable with density function given by $\ell(\alpha)$. Now, conditional on $\alpha$, the process is called a $\textit{conditional}$ or $\text{mixed}$  Bell-Touchard process with probability function of $N(t)$ given by:
\begin{equation}
\begin{aligned}[b]
    \Pr[N(t)=n]&=
    \int_{0}^{\infty}e^{-\alpha t(e^{\theta}-1)}\frac{\theta^n}{n!}B_{n}(\alpha t)\ell(\alpha)d\alpha
    \end{aligned}
\end{equation}
\end{definition}
\begin{proposition}
    If $\{N(t), t \geq 0\}$ is a conditional Bell-Touchard process and $\alpha\thicksim \text{Exp}(\gamma)$ then 
    \begin{equation}
        \Pr[N(t)=n]=\frac{\theta^n\gamma}{n!z} \text{Li}_{-n}(\frac{t}{z}),
    \end{equation}
where $z=[te^{\theta}+\gamma]$ and $\text{Li}_{-n}(t/z)$ is the polylogarithm function.
\end{proposition}
\begin{proof}
\begin{equation*}
\begin{aligned}[b]
    \Pr[N(t)=n]&=\int_{0}^{\infty}e^{-\alpha t(e^{\theta}-1)}\frac{\theta^n}{n!}B_{n}(\alpha t)\gamma e^{-\gamma\alpha}d\alpha\\
    &=\frac{\theta^n}{n!}\gamma\int_{0}^{\infty}e^{-\alpha t(e^{\theta}-1)-\gamma\alpha}B_{n}(\alpha t)d\alpha\\
    &=\frac{\theta^n}{n!}\gamma\int_{0}^{\infty}e^{-\alpha t(e^{\theta}-1)-\gamma\alpha}e^{-\alpha t}\sum_{k=0}^{\infty}\frac{k^n}{k!}(\alpha t)^k d\alpha\\
    &=\frac{\theta^n}{n!}\gamma\sum_{k=0}^{\infty}\frac{k^nt^k}{[te^{\theta}+\gamma]^{k+1}}\\
    &=\frac{\theta^n\gamma}{n!z} \text{Li}_{-n}(\frac{t}{z}),
\end{aligned}
\end{equation*}
writing $z=[te^{\theta}+\gamma]$ and using $\text{Li}_{-n}(t/z)$.
\end{proof}
\section{Numerical Results}

As it is well-known, stochastic processes are useful for the modeling of different phenomena. A way to do it is, after a suitable estimation of parameters, to perform simulations of the resulting stochastic process to predict the behavior of the phenomenon of interest. In what follows, we have the algorithm for the simulation of the Bell-Touchard process, which follows from the Proposition \ref{prop: bell_compound_poisson_family_member}. Consider $S_j\thicksim \text{ZTP}(\theta)$, ${\tau_j \thicksim \text{Exp}(\alpha(e^{\theta}-1))}$ and $\delta_n:=\sum_{i=0}^{n}\tau_i$, where $\tau_0=0 \text{ and } j \in \{0,1,2\dots\}$.

\SetKwComment{Comment}{/* }{ */}
\SetKwRepeat{Do}{do}{while}
\begin{algorithm}
\caption{Bell-Touchard process simulation}\label{alg:two}
\KwData{$\alpha, \theta \geq0$}
\KwResult{The random vectors T and P}
$\tau_0, \delta_0, S_0, i \gets 0,0,0,0$\;
Declare Arrays $T[0]\gets0 \text{ and } P[0]\gets0$\;
$\delta \gets \text{time length}$\;
 \Do{$\delta_i \leq \delta$}{
        $i\gets i+1$\;
        Generate $\tau_i$\;
        $\delta_{i} \gets \delta_{i-1}+\tau_i$\;
        \eIf{$\delta_i > \delta$}{
            Stop;
        }
        {
        $\text{T}[i] \gets \delta_i $\;
        Generate $S_i$\;
        $\text{P}[i] \gets S_i$\; 
        }
    }
   
\end{algorithm}


In order to illustrate the applicability of the Bell-Touchard process we consider the three data sets explored in \cite[Sections~4-5]{castellares2020two}. These data sets correspond, respectively, to: (1) the number of automobile insurance claims per policy over a fixed period; (2) the number of accidents of workers in a particular division of a large steel corporation within six months; and (3) the number of chromatid aberrations in 24 hours (see \cite{castellares2020two} and the references therein). It is not difficult to see that these three phenomena can be represented by counting processes. Moreover, since the Bell-Touchard distribution was fitted to these data, with the parameter estimates presented in \cite[Table~4]{castellares2020two}, we can use the Bell-Touchard process to simulate them. Thus, we shall take the parameter estimates and we run the respective path simulations of the Bell-Touchard process for each data set. Let $\Hat{\alpha} = (0.1760,0.2993, 0.4450)$ and $\Hat{\theta}=(0.3472,1.2667,0.6453)$, where $(\Hat{\alpha}_i,\Hat{\theta}_i)$ corresponds to the parameters estimates obtained by \cite{castellares2020two} of the i$^{\text{th}}$ data set, for $i\in\{1,2,3\}$. Before presenting the results we consider some additional definitions.


\begin{definition}
        Let $\{Y_{i}\}_{i=1}$ be a sequence of i.i.d. random variables where $Y_{\ell} \thicksim \Gamma(\eta,\beta)$ for all $\ell \in \mathbb{N}$ and ${\{N(t),t\geq0\}}$ a Bell-Touchard process with parameters $(\alpha,\theta)$. Then
        \begin{equation}
            \mathcal{L}(t) = \sum_{\ell=1}^{N(t)}Y_{\ell}, \text{ } t\geq 0
            \label{eq:bell_compound_processes}
        \end{equation}
        is a compound Bell-Touchard process with $\mathbb{E}[\mathcal{L}(t)] = \alpha t\theta e^{\theta}\frac{\eta}{\beta}$.
        \label{def:bell_compounding_processes}
\end{definition}

Now, consider the modification of the classic surplus process\cite{asmussen2010ruin}:

\begin{equation}
    \mathcal{R}_{t} = u+\rho t-\mathcal{L}(t),
\end{equation}
and define 
\begin{equation*}
    \rho_{\varepsilon} := (1+\varepsilon)\alpha \theta e^{\theta}\frac{\eta}{\beta}   
\end{equation*}
where $\varepsilon\geq0$ is the safety loading. As a result, one has 
\begin{equation*}
    \mathbb{E}[\mathcal{R}_t] = u+\varepsilon\alpha \theta e^{\theta}\frac{\eta}{\beta}.
\end{equation*}
Therefore, we have the following process
\begin{equation}
    \mathcal{R}_t = u+\rho_{\varepsilon}t-\mathcal{L}(t),
    \label{eq:ruin_modification}
\end{equation}
whose the number of claims follows the Bell-Touchard process.

By using the Algorithm \ref{alg:two} one can simulate the paths of the Bell-Touchard process for each of the data sets mentioned before. Furthermore, we generate the simulation for the process \eqref{eq:ruin_modification} using the parameter estimates of the first data set mentioned (see Figure \ref{label2}). Looking at Figure \ref{label1}, one can see the path of the Bell-Touchard process corresponding to the automobile claims data set (1). 
\begin{figure}[htb]
\centering
\hspace*{1.5em}\raisebox{\dimexpr-.5\height-1em}
  {\includegraphics[scale=0.45]{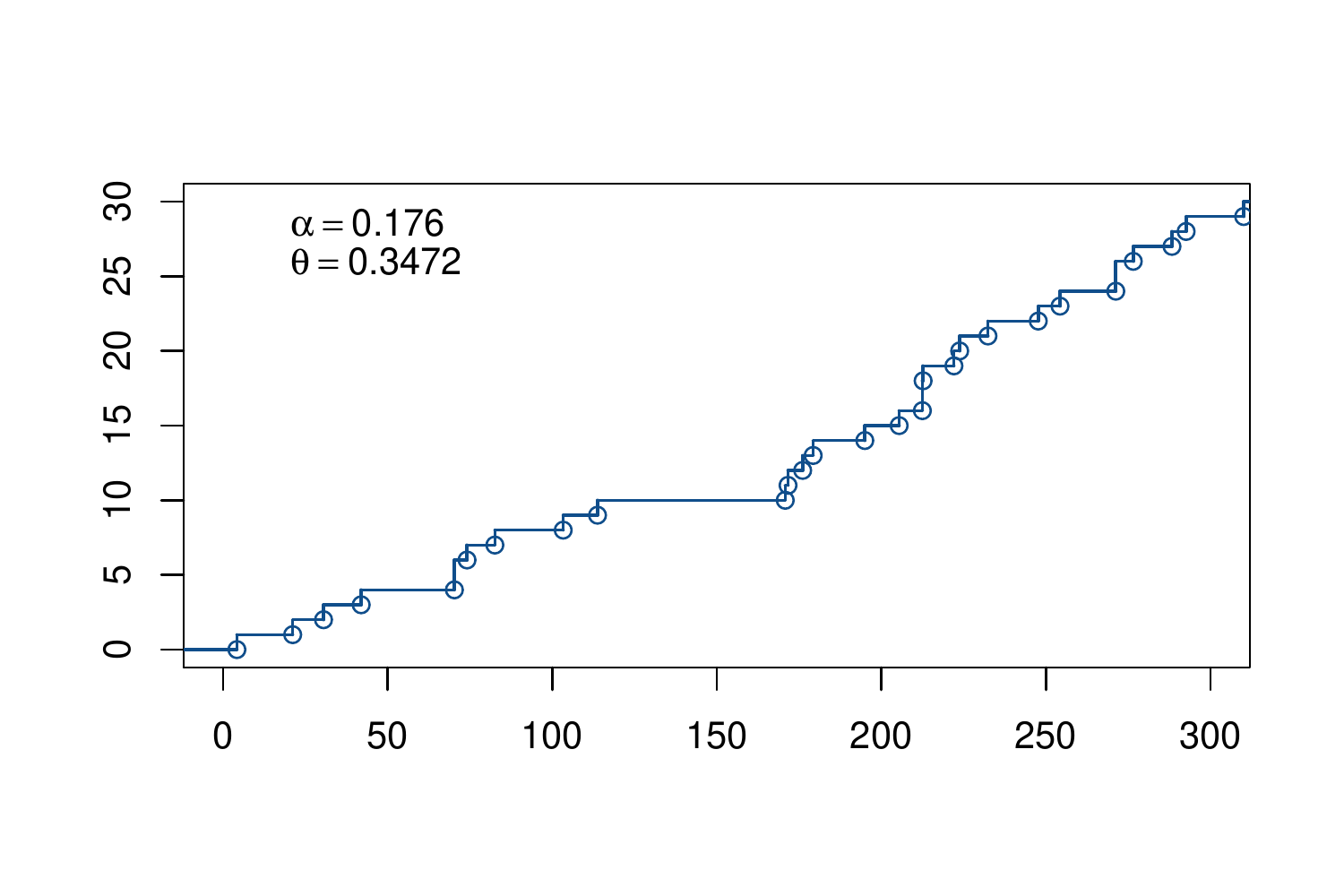}}%
\subfloat[\label{label1}]{\hspace{1.5em}}
\hspace*{1.5em}\raisebox{\dimexpr-.5\height-1em}
  {\includegraphics[scale=0.45]{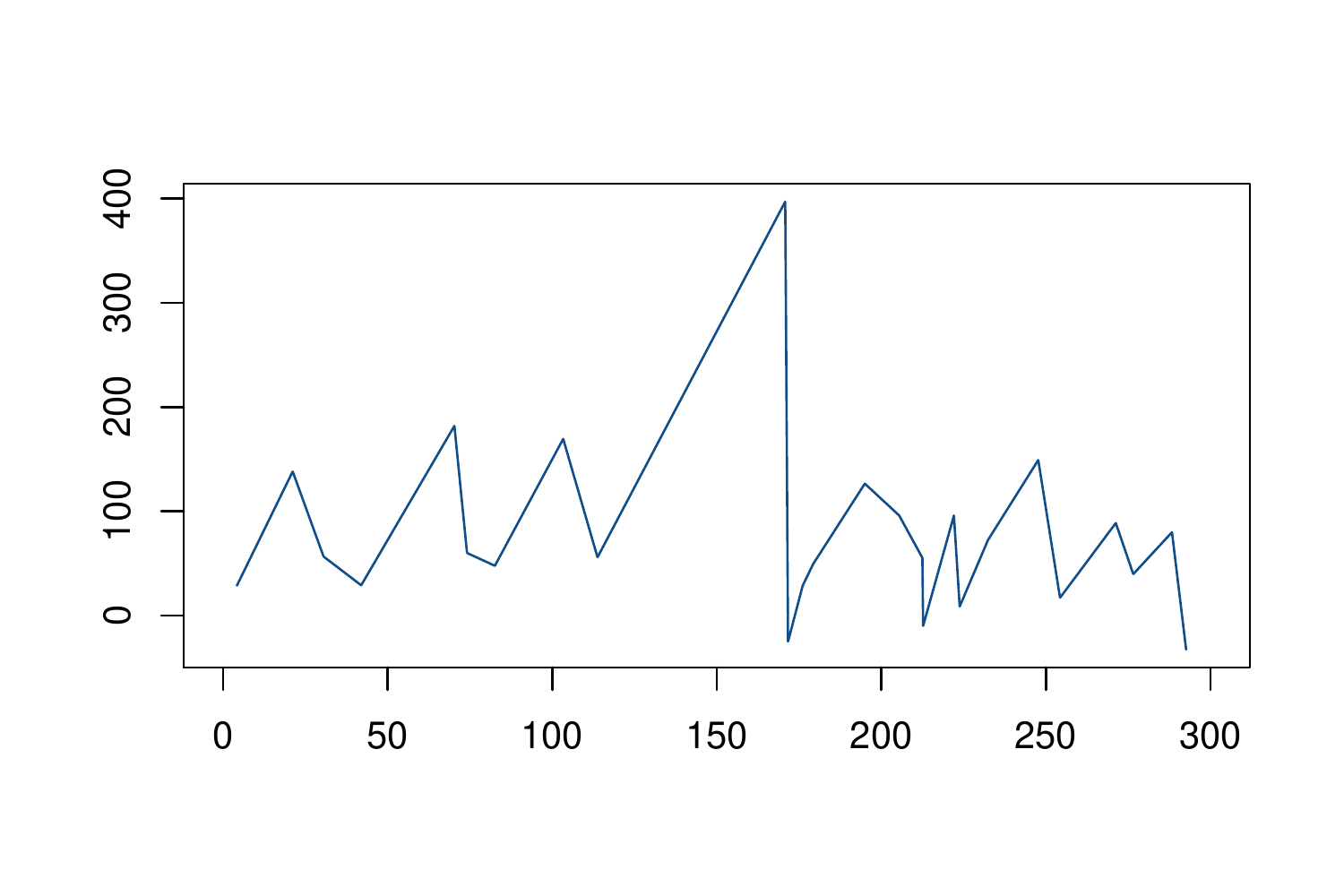}}%
\subfloat[\label{label2}]{\hspace{1.5em}}

\leavevmode\smash{\makebox[0pt]{\hspace{-37em}
  \rotatebox[origin=l]{90}{\hspace{4.5em}
    $N(t)$}%
}}
\leavevmode\smash{\makebox[0pt]{\hspace{2em}
  \rotatebox[origin=l]{90}{\hspace{4.5em}
    $R(t)$}%
}}

Time - $t$

\medskip

\caption{The Bell-Touchard process and modified risk process for data set (1). 
}
\end{figure}

The path in the Figure \ref{label2} depends on the path of Figure \ref{label1} and the intensity of the claims amount pictured in Figure \ref{label3}. Inside the Figure \ref{fig:paths} one can see two-path simulations for both data sets 1 and 2.

\begin{figure}[htb]
\centering
\hspace*{1.5em}\raisebox{\dimexpr-.5\height-1em}
  {\includegraphics[scale=0.60]{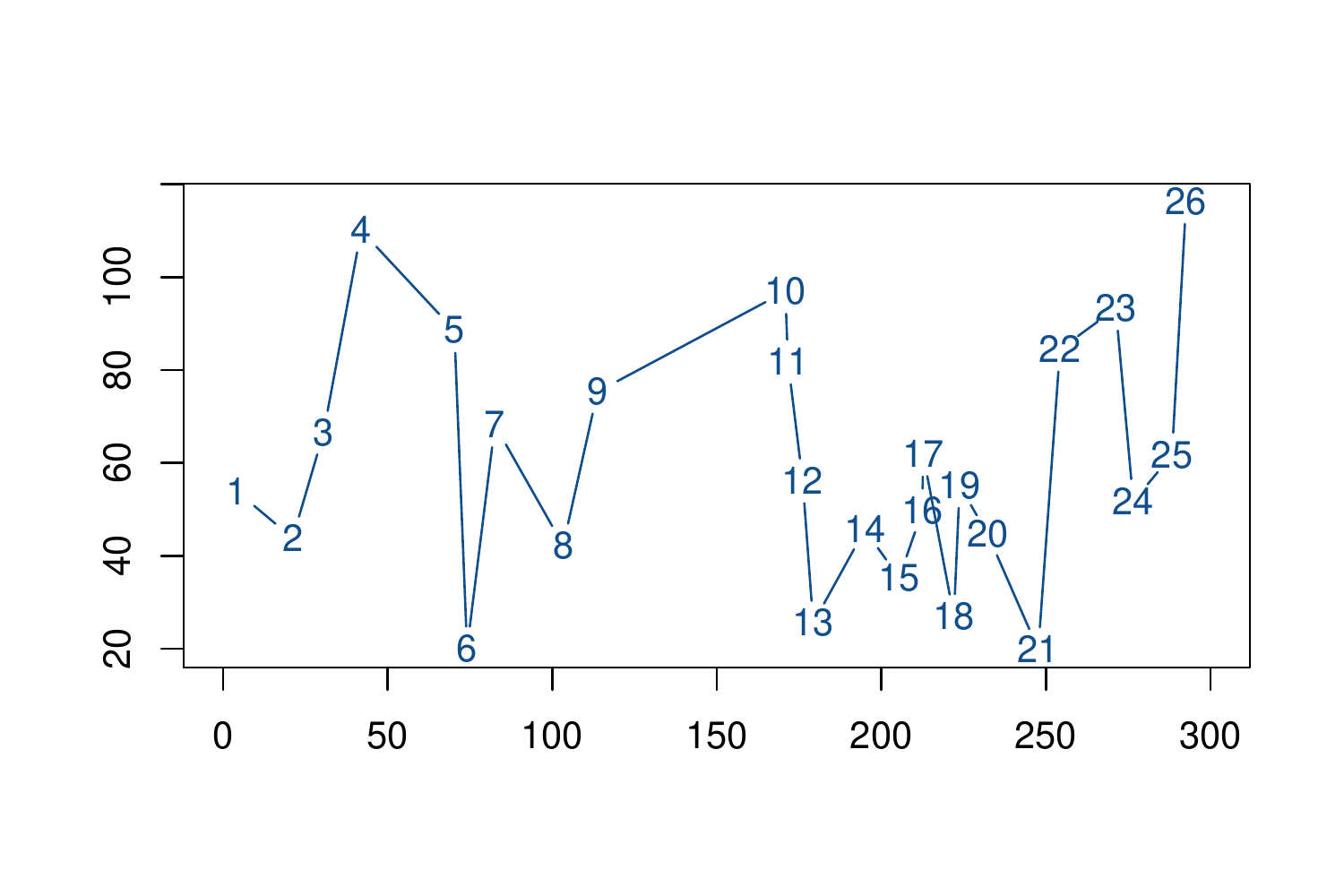}}%
\subfloat[\label{label3}]{\hspace{1.5em}}

\leavevmode\smash{\makebox[0pt]{\hspace{18em}
  \rotatebox[origin=l]{90}{\hspace{4em}
    Claim Amount}%
}}\hspace{0pt plus 1filll}\null

Time - $t$


\caption{The claim amount per occurrence throughout the modified process in Figure \ref{label2}. 
}
\label{fig:claim_amount}
\end{figure}

\begin{figure}[htb]
\centering
\hspace*{1.5em}\raisebox{\dimexpr-.45\height-1em}
  {\includegraphics[scale=0.45]{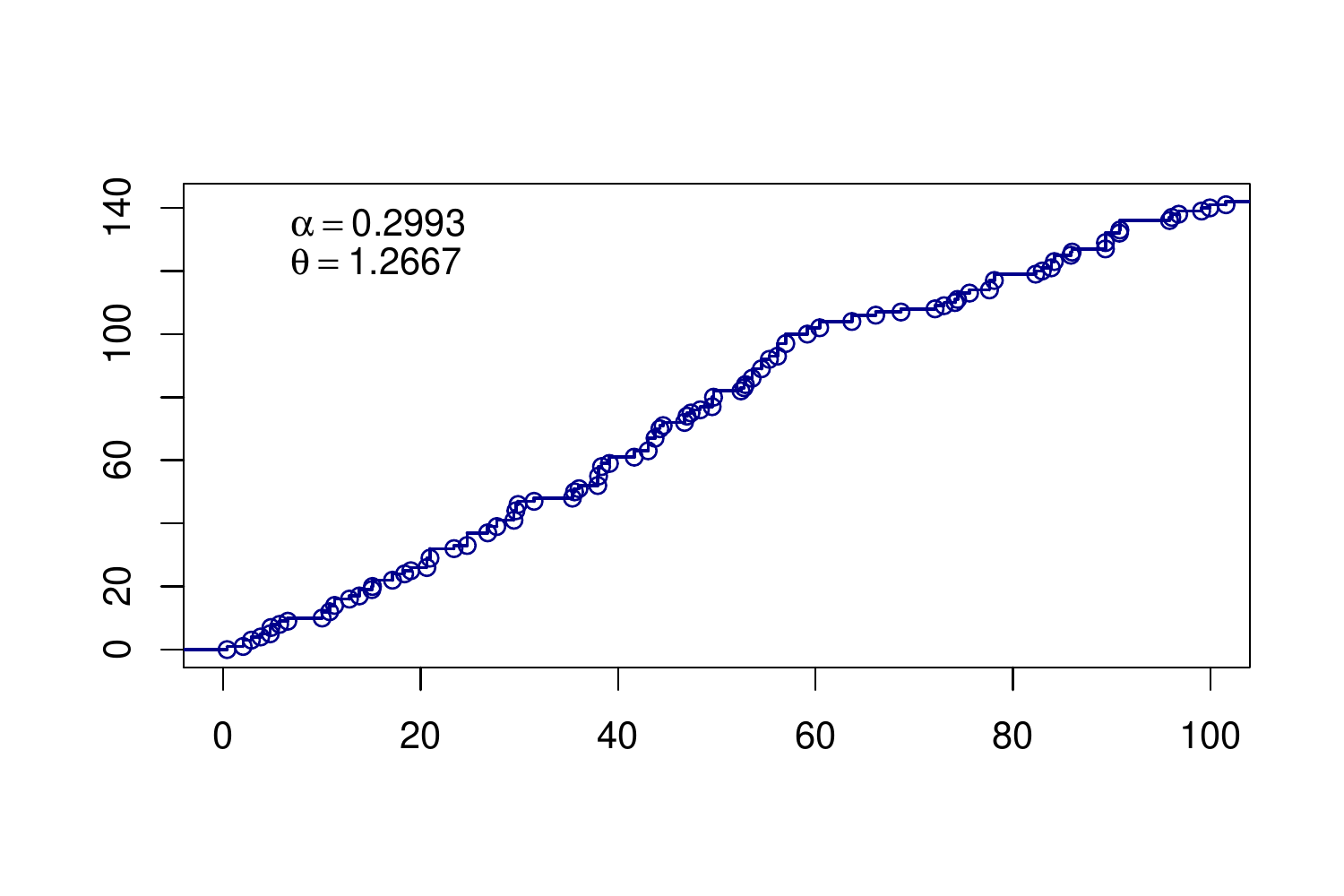}}%
\subfloat[\label{label11}]{\hspace{1.5em}}
\hspace*{1.5em}\raisebox{\dimexpr-.45\height-1em}
  {\includegraphics[scale=0.45]{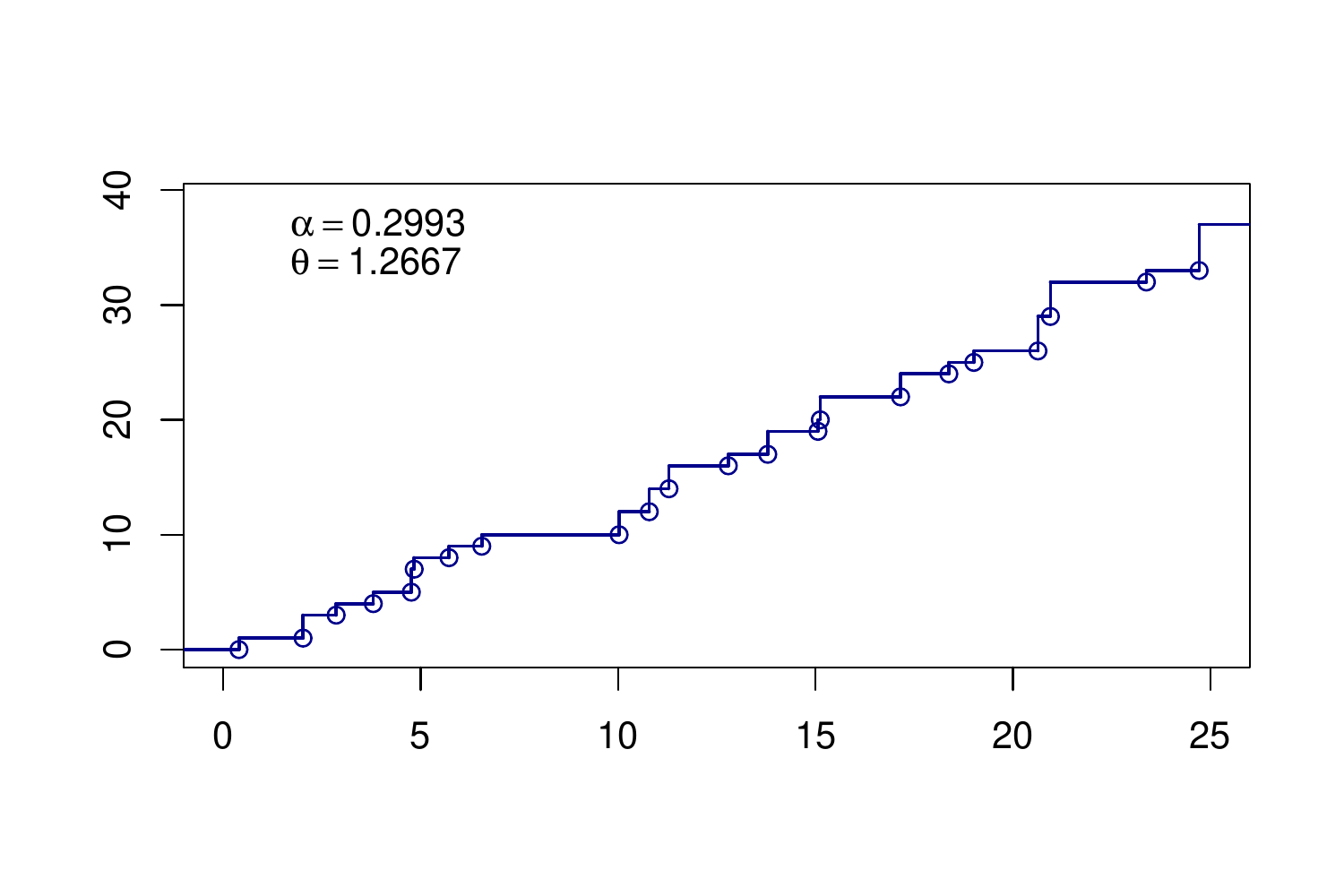}}%
\subfloat[\label{label22}]{\hspace{1.5em}}
\\[\smallskipamount]
\hspace*{1.5em}\raisebox{\dimexpr-.45\height-1em}
  {\includegraphics[scale=0.45]{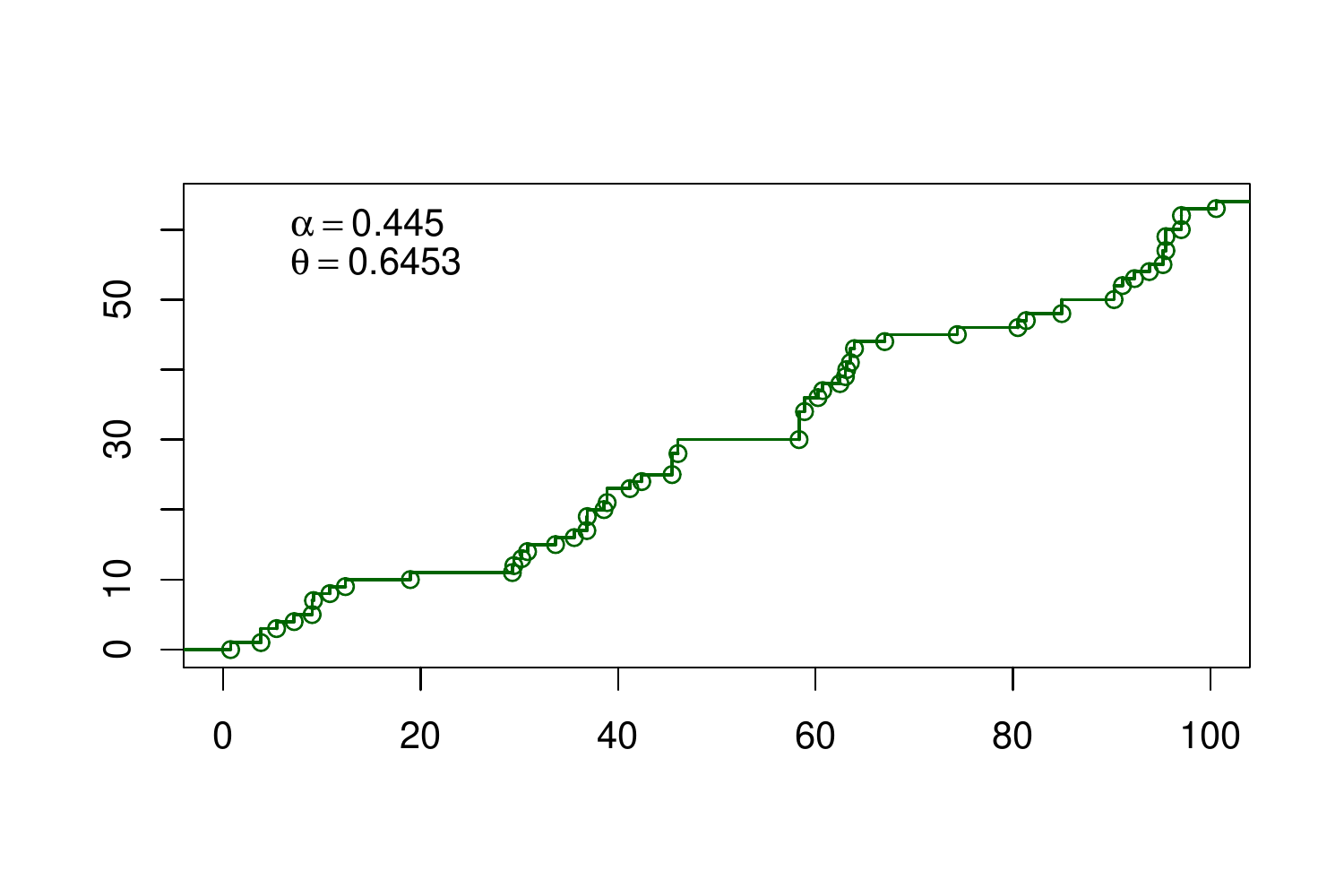}}%
\subfloat[\label{label12}]{\hspace{1.5em}}
\hspace*{1.5em}\raisebox{\dimexpr-.45\height-1em}
  {\includegraphics[scale=0.45]{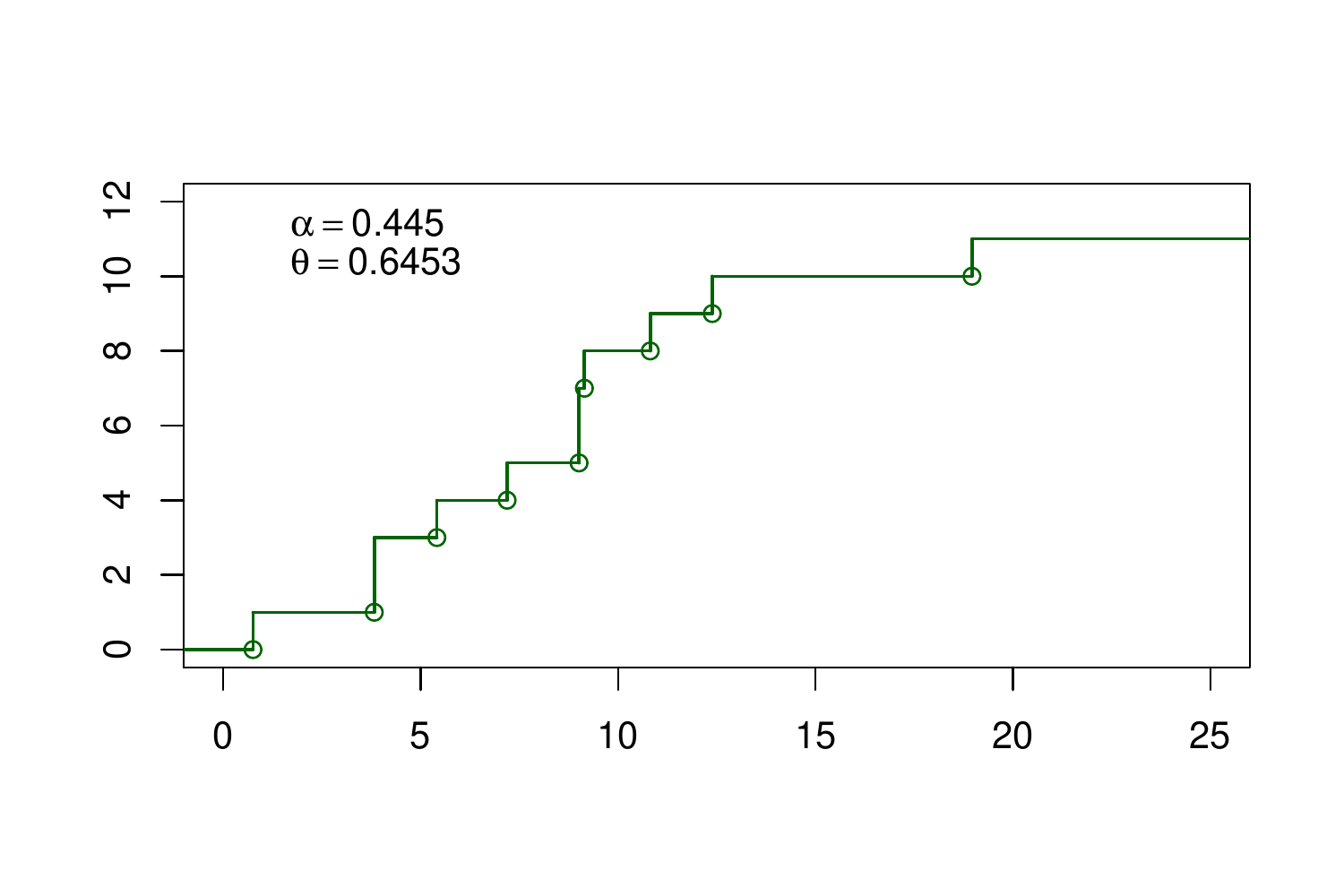}}%
\subfloat[\label{label21}]{\hspace{1.5em}}

\leavevmode\smash{\makebox[0pt]{\hspace{2em}
  \rotatebox[origin=l]{90}{\hspace{10em}
    $N(t)$}%
}}\hspace{0pt plus 1filll}\null

Time - $t$

\medskip

\caption{The Bell-Touchard process path simulation for data sets (2) and (3). 
}
\label{fig:paths}
\end{figure}



\section{Concluding Remarks}
We have proposed a counting process based on the Bell-Touchard distribution, called the Bell-Touchard process. To reach that, we have explored some relevant properties of the Bell polynomials and the Bell-Touchard probability distribution. Moreover, we have shown this new process is a member of the multiple Poisson process family and a compound Poisson process. In addition, we have found that the Bell-Touchard process is closed for superposition and decomposition operations. However, despite the Bell-Touchard process being closed for decomposition procedures, the resulting processes are non-independent Bell-Touchard processes. Furthermore, we have found that the iterated Poisson process is a Bell-Touchard process. In addition, we have proposed a generalization of the process, naming it the nonhomogeneous Bell-Touchard process. We had shown that a nonhomogeneous Bell-Touchard process results from the composition of Poisson processes. In this case, however, the inside process is a nonhomogeneous Poisson process.

One can apply the Bell-Touchard process in many situations, such as risk modelling of catastrophic events, queue theory and ruin theory. In principle, any circumstance demanding the generalization of the Poisson process to provide any size jumps with a mathematically tractable underlying probability distribution could be a suitable application. An application in ruin theory could be generalizing the classic risk process, changing the compound Poisson process by a sum of positive random variables with the number of variables summed up following a Bell-Touchard process. In our work we appeal to three data sets from literature to illustrate the applicability of our process. 

In the nonhomogeneous Bell-Touchard process, we set the parameter $\alpha(t)$ for $t\geq0$ as a function depending on time. Considering that, as future research, we suggest setting as parameter the function $\theta(t)$, for $t\geq0$. In this case, the resulting process would have jumps intensity driven by a function $\theta:t\rightarrow [0,\infty)$. An advantage of this approach would be finding a process whose jumps size was not stationary. In addition, considering the modified risk process previous mentioned, finding an upper bound for the ruin probability concerning it, maybe through a modification of the Lundberg bound,  would be an interesting research problem.

\section{Acknowledgements}

This study was financed in part by the Coordena\c{c}\~ao de Aperfei\c{c}oamento de Pessoal de N\'ivel Superior - Brasil (CAPES) - Finance Code 001. The second author thanks also FAPESP (grant 17/10555-0). The authors thank Elcio Lebensztayn (UNICAMP) and Valdivino Vargas Junior (UFG) for fruitful discussions. 
\printbibliography
\end{document}